\documentclass[11pt,reqno]{amsart}
\usepackage{graphicx}
\usepackage{amsfonts,amsmath,amssymb}
\def\1#1{\overline{#1}}
\def\2#1{\widetilde{#1}}
\def\3#1{\widehat{#1}}
\def\4#1{\mathbb{#1}}
\def\5#1{\frak{#1}}
\def\6#1{{\mathcal{#1}}}

\def\C{{\4C}}

\def\Z{{\4Z}}

\renewcommand{\Im}{\tmop{Im}}

\renewcommand {\a}{\alpha}
\renewcommand {\b}{\beta}

\newtheorem{thm}{Theorem}[section]
\newtheorem{propos}[thm]{Proposition}
\newtheorem{corol}[thm]{Corollary}

\theoremstyle{definition}
\newtheorem{dfn}[thm]{Definition}

\newcommand{\tr}{\ensuremath{\mbox{\bf tr}\,}}

\renewcommand{\Im}{\mathop{\rm Im}\nolimits}
\newcommand{\im}{\ensuremath{\mbox{\rm Im}\,}}
\newcommand{\re}{\ensuremath{\mbox{\rm Re}\,}}

\newcommand{\theor}[1]{\smallskip  \noindent \bf Theorem #1.\it\,\,}

\newcommand{\CC}[1]{\mathbb{C}^{#1}}

\newcommand{\RR}[1]{\mathbb{R}^{#1}}

\numberwithin{equation}{section}

\title{Convergent normal form for real hypersurfaces at generic Levi degeneracy }
\author {I. Kossovskiy}
\address{Department of Mathematics, University of Vienna}
\email{ilya.kossovskiy@univie.ac.at}
\thanks{*Supported by the Austrian Science Foundation grant M1413-N25.}
\author {D. Zaitsev*}
\address{School of Mathematics, Trinity College, Dublin}
\email{zaitsev@maths.tcd.ie}
\thanks{*Supported in part by the Science Foundation Ireland grant 10/RFP/MTH2878.}

\begin{document}

\date{\today}

\begin{abstract}
We construct a complete convergent normal form for a real
hypersurface in $\CC{N},\,N\geq 2$ at generic Levi degeneracy.
This seems to be the first convergent normal form for a
Levi-degenerate hypersurface. In particular, we obtain, in the
spirit of the work of Chern and Moser \cite{chern}, distinguished
curves in the Levi degeneracy set, that we call \it degenerate
chains.\rm
\end{abstract}

\maketitle \tableofcontents

\section{Introduction}
In the study of any geometric structure on a manifold $M$, a {\em
normal form}, corresponding to special  choices of coordinates
adapted to the structure, is of fundamental importance. In case of
Riemannian metric, such special coordinates are the normal
coordinates given by the exponential map. For a vector field $X$
that does not vanish at a point $p\in M$, special local
coordinates can be chosen, where this vector field is constant.
If, however, $X$ does vanish at $p$, its normal form, known as
{\em Poincare-Dulac normal form} \cite{ilyashenko}, exists in
general only in formal sense, whereas its convergence is a
delicate issue depending on presence of so-called small divisors.
Because of such clear difference in behavior, a point where $X$
vanishes, is to be treated as {\em singularity of $X$}.

The study of real submanifolds $M$ in complex spaces $\C^n$ is
remarkable in which it exhibits both regular and singular
phenomena. A simple example of a regular point is a {\em CR point}
$p_0\in M$, for which the {\em complex tangent space}
$$T_p^\C:= T_pM \cap JT_pM$$
is of constant dimension for $p\in M$ near $p_0$, whereas a point
$p_0$ for which this dimension is not constant in any
neighborhood, is called a {\em CR singularity}. (Here $J$ is the
standard complex structure on $\C^n$.) In relation with normal
form (and many other questions), a regular behavior occurs for
{\em Levi-nondegenerate hypersurfaces}, where a normal form was
constructed by Chern and Moser \cite{chern}. On the other hand, at
``singular" points, where the Levi form is degenerate, no such
normal form is known in general.

A particular property of the normal form constructed by Chern and
Moser is its {\em convergence} for $M$ real-analytic. As a
consequence, geometry of the CR structure of $M$ can be studied
using its normal form. Another remarkable consequence is the
presense of so-called {\em chains}, i.e.\ certain distinguished
real curves that can be locally constructed as certain lines in
coordinates corresponding to the normal form.

Since the work of Chern and Moser, normal forms have been
constructed {\em in the formal sense} for certain classes of
Levi-degenerate hypersurfaces by Ebenfelt \cite{ebenfeltC3,
ebenfeltjdg}, Wong \cite{wong}, Kolar \cite{kolar}, Kolar and
Lamel \cite{kl}. We also refer to Moser-Webster \cite{mw}, Huang
and Yin \cite{hy}, and Burcea \cite{bu} for normal forms at
CR-singular points. These forms were either known to be divergent
in general or the question about their convergence was left open.
See e.g.\ \cite{mw} where convergence only holds in so-called {\em
elliptic case} or \cite{kolardiverg} for examples of divergent
normal forms. For normal forms for Levi-nondegenerate CR-manifolds
in higher codimension see Ezhov and Schmalz \cite{es} and
Beloshapka \cite{obzor}.

\medskip
The goal of this paper is to establish what seems to be the first
convergent normal form for Levi-degenerate real-analytic
hypersurfaces. Our normal form is constructed for the well-known
class of {\em generic Levi-degeneracy} points introduced by
Webster \cite{webstergeneric}, i.e.\ points where the determinant
of the Levi form vanishes but its differential restricted to the
complex tangent space doesn't. Note that a point of generic Levi
degeneracy is ``stable'' in the sense that it cannot be removed by
a small perturbation. A different {\em formal normal form} for
generic Levi degeneracy points, whose convergence remains unknown,
was constructed by Ebenfelt \cite{ebenfeltjdg}. More recently,
Kolar \cite{kolar} constructed a formal normal form for all
hypersurfaces of finite type in $\C^2$, which, however, is
divergent in general (see \cite{kolardiverg}).

The convergence proof for the normal form by Chern and Moser is
heavily based on the property that  Levi forms at different points
are isomorphic. In fact, the geometry and normal forms look
similar at all points. Consequently, normalization conditions for
the normal form at a point $p\in M$ depend analytically on $p$,
and hence, can be translated into systems of certain analytic
ordinary differential equations whose solutions are again
real-analytic. This is {\em not the case any more}
 for points of generic
Levi degeneracy, in whose neighborhoods the Levi form does not
have constant rank. Thus geometry at those degenerate points is
fundamentally different from that at Levi-nondegenerate points. To
overcome this new difficulty, we first restrict to the subset of
all Levi-degenerate points of $M$ which, at a point of generic
Levi degeneracay, is always a (real-analytic) hypersurface
$\Sigma$ in $M$, transverse to the Levi kernel at each point of
it. Then we set up a system of real-analytic ordinary differential
equations {\em along $\Sigma$} defining analogues of Chern-Moser
chains, that we call {\em degenerate chains}, and further
differential equations to restrict (normalize) parametrizations of
degenerate chains. Through every point of $\Sigma$ we obtain a
unique (in the sense of germs) degenerate chain, and its
normalized parametrization is determined up to linear scaling.
Also note that our degenerate chains can be still defined for
merely smooth real hypersurfaces by the same equations.

\medskip
We now give complete statements of our normal form, which are
different in $\C^2$ and higher dimension. This is due to the
presense of Levi-nondegenerate directions in higher dimension. We
write
$$(z,w)=(z, u+iv)\in \C\times \C, \quad
(\check z, z_n, w) = ( z_1,\ldots, z_{n-1}, z_n, u+iv)\in
\C^{n-1}\times \C\times \C, $$ for the coordinates in $\C^2$ and
in $\C^{n+1}$, $n\ge 2$, respectively. Then in $\C^2$ we obtain:

\theor{1} Let $M\subset\CC{2}$ be a real-analytic hypersurface
with generic Levi degeneracy at a point $p\in M$. Then there
exists a local biholomorphic transformation $F$ of $\C^2$ sending
$p$ to $0$ and $M$ into a normal form
\begin{equation}
\label{normalform} v=2\re(z^2\bar z)+\sum\limits_{k,l\geq
2}\Phi_{kl}(u)z^k\bar z^l, \quad
\re\Phi_{32}(u)=\im\Phi_{42}(u)=0.
\end{equation}

The germ of $F$ at $p$ is uniquely determined by the restriction
of its differential $dF$ to the complex tangent space
$T^{\C}_{p}M$.

Furthermore, the Levi degeneracy set of $M$ is canonically
foliated by degenerate chains, where the chain through $p$ is
locally given by $z=0$ in any normal form at $p$. \rm

\bigskip

The normal form here is similar but different from the formal
normal form of Kolar \cite{kolar}.

In the case of higher dimension, we write $M$ as
$$v
=\sum_{k',k,l',l\in\Z_{\ge0}} \Phi_{k'kl'l}(\check z,\1{\check
z},u) z_n^k \bar z_n^l,$$ where each $\Phi_{k'kl'l}(\check
z,\1{\check z},u)\in \C\{u\}[\check z,\1{\check z}]$ is a
bi-homogeneous polynomial of bi-degree $(k',l')$ in $(\check
z,\overline{\check z})$ with real-analytic coefficients in $u$.
Denote by $r\geq\frac{n-1}{2}$ the number of positive eigenvalues
of the Levi form of $M$ at a reference point $p\in M$, and write
\begin{equation}\label{Levi}
\langle \check z,\overline{\check
z}\rangle:=\sum\limits_{j=1}^{n-1}\varepsilon^j\bigl|z_j\bigr|^2,
\quad \varepsilon^j=1 \text{ for }  j\leq r, \quad
\varepsilon^j=-1 \text{ for } j>r,
\end{equation}
and
$$\mbox{\bf tr}:=\sum\limits_{j=1}^{n-1}\varepsilon^j\frac{\partial^2}{\partial z_j\partial\overline{z_j}}$$
for the {\em trace operator} (that is also used by Chern and Moser
in their normal form). Now our main result in higher dimension is
the following:

\theor{2} Let $M\subset\CC{n+1}$, $n\geq 2$, be a real-analytic
hypersurface with generic Levi degeneracy at a point $p\in M$.
Then there exists a local biholomorphic transformation $F$ of
$\C^{n+1}$ sending $p$ into $0$ and $M$ into a normal form
\begin{equation*}
v=\left\langle \check z,\overline{\check z}\right\rangle +
2\re\bigl(z_n^2\overline{z_n}\bigr) +\sum\limits_{k'\geq
2}\bigl(\Phi_{k'001}(\check z,\1{\check z},u)\bar z_n
+\Phi_{01k'0}(\check z,\1{\check z},u)z_n\bigr)\, +\sum\limits_{
k'+k, l'+l\geq 2}\Phi_{k'kl'l}(\check z,\1{\check z},u) z_n^k \bar
z_n^l,
\end{equation*}
where terms in the last sum satisfy the normalization conditions
\begin{eqnarray}
\label{normalformn}
 \Phi_{1102}=0,\quad \tr\Phi_{1111}=0,\notag\\
\re\Phi_{1211}=0,\quad \im\tr\Phi_{1211}=0.
\end{eqnarray}

The germ of $F$ at $p$ is uniquely determined by the restriction
of its differential $dF$ to the complex tangent space
$T^{\C}_{p}M$, which in turn, is uniquely determined by a pair
$(\lambda,C)$, where $\lambda$ is a real scalar and $C$
is a complex-linear automorphism of $\C^{n-1}$ such that
$$\langle C\check z,\overline{C\check z}\rangle = \lambda
\langle \check z,\overline{\check z}\rangle.$$

Furthermore, the Levi degeneracy set of $M$ is canonically
foliated by degenerate chains, where the chain through $p$ is
locally given by $z=0$ in any normal form at $p$. \rm

\bigskip

Note that our normal form here is very different from
\cite{ebenfeltjdg}, and also depends on fewer parameters. In fact,
the normal form is parametrized precisely by the group of all
automorphisms (fixing the origin) of the model hypersurface
\begin{equation}\label{model0}
v=\left\langle \check z,\overline{\check z}\right\rangle +
2\re\bigl(z_n^2\overline{z_n}\bigr),
\end{equation}
which, as a consequence of Theorem~2, consists of all linear
automorphisms of $\CC{n+1}$, preserving \eqref{model0}.

\section{Notations and preliminaries}
Recall that the {\em Levi form} of a real hypersurface
$M\subset\CC{n+1},\,n\geq 1$ is defined for CR points $p$ by
$$\6L_p\colon T_p^\C M \times T_p^\C M \to
\C\otimes (T_pM/ T_p^\C M), \quad \6L_p(X( p),Y( p)) = [X^{10},
Y^{01}]( p)
 \mod \C\otimes T_pM,
 $$
where $X$ and $Y$ are vector fields in $T^\C M$ and
$$X^{10}:=X-iJX, \quad X^{01}:= X +iJX,$$
are the corresponding $(1,0)$ and $(0,1)$ vector fields.

For a real-analytic hypersurface $M\subset\CC{2}$ with generic
Levi degeneracy at a point $p\in M$, we consider its Levi
degeneracy set $\Sigma$, which for $M\subset\CC{2}$ is simply the
set of points where the whole Levi form is zero. In our case of
generic Levi degeneracy $\Sigma$ is a totally real submanifold.
 Then for
$p\in \Sigma$ we can naturally define, using third order Lie
brackets, the canonical cubic form
\begin{equation}\label{cubicform}
c\colon (T^{\C}_{p}M)^{3} \to \C\otimes\frac{T_{p}M}{T^{\C}_{p}M},
\quad c(X(p), Y(p), Z(p)) = [X^{10}, [Y^{10}, Z^{01}]](p) \mod
\C\otimes T_{p}^{\C}M.
\end{equation}
We can always choose local coordinates in $\CC{2}$ in such a way
that
$$p=0,\quad T_0M=\{v=0\},\quad T_0\Sigma=\{v=0,\,\re z=0\},$$
 and the
canonical cubic form at $0$ is given by
$z^2\bar z$.
In what follows we always assume that coordinates for $M$ are
chosen in this manner. Considering a hypersurface, given near the
origin by the defining equation $v=\Phi(z,\bar z,u)$, and using
expansions of the form
$$\Phi(z,\bar z,u)=\sum\limits_{k,l\geq 0}\Phi_{kl}(u)z^k\bar
z^l,$$ where $\Phi_{kl}(u)$ are real-analytic near the origin, we
can read the above normalization of the canonical cubic form  as
$\Phi_{21}(0)=1$.

Let then $M\subset\CC{n+1},\,n\geq 2$, be a real-analytic
hypersurface with generic Levi degeneracy at a point $p\in M$. It
is shown in \cite{ebenfeltjdg} that one can define for $M$ a
third-order invariant cubic form
\begin{equation}\label{cubicformn}
c\colon T^{10}_pM \times T^{10}_pM \times K^{01}_p  \to \C\otimes\frac{T_{p}M}{T^{\C}_{p}M},
\end{equation}
which we call \it Ebenfelt's tensor, \rm such that its restriction
onto $(K_{p}^{10})^{2}\times K_p^{01}$ is non-vanishing (here
$K_p\subset T^{\CC{}}_pM$ is the Levi kernel of $M$ at $p$,
$\mbox{dim}_{\CC{}}K_p=1$). The canonical cubic form
\eqref{cubicformn} replaces the tensor \eqref{cubicform} in the
higher dimensional case and enables us to fix canonically the \it
Levi-nondegenerate direction at $p$ \rm as the complex hyperplane
$$K^T_p=\Bigl\{X\in T^{\CC{}}_pM:\quad c(X^{10},Y^{10},Z^{01})=0\quad \forall\, Y,Z\in K_p\Bigr\}\subset T^{\CC{}}_pM.$$
The subspace $K^T_p$ is transverse to $K_p$. We assume in what
follows that local coordinates for $M$ near $p$ are chosen in such
a way that $p$ is the origin, the tangent space at $0$ is
$\{v=0\}$, the Levi kernel at $0$ is given by $\{\check
z=0,\,w=0\}$, the Levi-nondegenerate direction by
$\{z_n=0,\,w=0\}$, the Levi form at $0$
 by
\begin{equation}\label{Levi}
\langle \check z,\overline{\check
z}\rangle:=\sum\limits_{j=1}^{n-1}\varepsilon^j\bigl|z_j\bigr|^2,
\quad \varepsilon^j=1 \text{ for } 0\leq j\leq r, \quad
\varepsilon^j=-1 \text{ for } r+1\leq j\leq n-1,
\end{equation}
and Ebenfelt's tensor by $(z_n)^2\overline{z_n}.$ The canonical
Hermitian form (i.e., the Levi form) and the canonical cubic form,
associated with $M$, enable us to define the differential operator
$$\mbox{\bf tr}:=\sum\limits_{j=1}^{n-1}\varepsilon^j\frac{\partial^2}{\partial z_j\partial\overline{z_j}}$$
on the space of formal series $\Phi(\check z,\1{\check z},u)$.
We have the identity
$$\mbox{\bf tr}\langle\check z,\overline{\check
z}\rangle=n-1.$$

We are now in the position to proceed with the normalization
procedure. As  presence of the Levi-nondegenerate direction for a
hypersurface in $\CC{n+1},\,n\geq 2$ makes the normalization
procedure  significally different from the one in the
two-dimensional case, we consider the two-dimensional case
separately in Section 3, and then consider the high-dimensional
case in Section 4.

\section{The two-dimensional case}

Let $M\subset\CC{2}$ be a real-analytic hypersurface with generic
Levi degeneracy at the point $0\in M$. Choosing local holomorphic
coordinates $(z,w)=(z,u+iv)$ near the origin as described before,
we represent $(M,0)$ by a defining equation
$$v=\Phi(z,\bar z,u)=\sum\limits_{k,l\geq 0}\Phi_{kl}(u)z^k\bar
z^l$$ with $\Phi(0)=0,\,d\Phi(0)=0,\
\Phi_{11}(0)=0,\,\Phi_{21}(0)=1$.  We perform a transformation
$w\mapsto w+Q(z)$, where $Q(z)$ is a cubic polynomial, in order to
eliminate the pure quadratic and pure cubic terms in $\Phi$. Hence
we end up with a hypersurface
\begin{eqnarray}
 \label{simplified} v=\Phi(z,\bar
z,u)=\sum\limits_{k,l\geq 0}\Phi_{kl}(u)z^k\bar z^l,\\
\Phi(0)=\Phi_{10}(0)=\Phi_{11}(0)=\Phi_{20}(0)=\Phi_{30}(0)=0,\,\Phi_{21}(0)=1.\notag
\end{eqnarray}
In what follows we consider only transformations, preserving the
form \eqref{simplified}. The Levi-degeneracy set $\Sigma\subset M$
is a totally real manifold
\begin{equation}
\label{nonflat} \Sigma=\bigl\{\re z=\chi(\im z,u),\quad\im w
=\tau(\im z,u)\bigr\},\quad \chi(0)=\tau(0)=0,\,d\chi(0)=d\tau(0)=
0.
\end{equation}

\subsection{Formal normalization in the two-dimensional case}

\mbox{}\\

We start with the proof of the formal version of Theorem 1. Let us
introduce weights for the variables $(z,w)$ in the following way:
$$[z]=1,\,[w]=3.$$
These weights correspond to the choice of a \it model \rm for
hypersurfaces of the form \eqref{simplified}, given by
\begin{equation}
\label{model} v=2\re(z^2\bar z).
\end{equation}
This model, first introduced in \cite{belsmall} and later used in
\cite{kolar}, is weighted homogeneous with respect to this choice
of weights. Then for any formal power series
$f(z,w)\in\CC{}[[z,w]]$ without constant terms we get the formal
expansion
$$f(z,w)=\sum_{m\geq 1}f_m(z,w),$$
 where each $f_m(z,w)$ is a
weighted homogeneous polynomial of weight $m$. Similarly, for the
right-hand side we get the expansion
$$\Phi=2\re(z^2\bar z)+\sum_{m\geq 4}\Phi_m(z,\bar z,u)$$
(so that any hypersurface \eqref{simplified} can be interpreted as
a perturbation of the model \eqref{model}).

If now $M=\{v=\Phi(z,\bar z,u)\}$ and $M^*=\{v^*=\Phi^*(z^*,\bar
z^*,u^*)\}$ are two hypersurfaces, satisfying \eqref{simplified},
and $z^*=f(z,w),\,w^*=g(z,w)$ is a formal invertible
transformation, transforming $(M,0)$ into $(M^*,0)$, we obtain the
identity
\begin{equation}
\label{tangency} \im g(z,w)|_{w=u+i\Phi(z,\bar z,u)}
=\Phi^*(f(z,w),\overline{f(z,w)},\re g(z,w))|_{w=u+i\Phi(z,\bar
z,u)}.
\end{equation}
Collecting in \eqref{tangency} terms of the low weights $1,2$, we
get $g_1=g_2=0$. For the weight $3$ we obtain $f_1=\lambda
z,\,g_3=\lambda^3 w,\,\lambda=f_z(0,0)\in\RR{}\setminus\{0\}$.
Thus the initial transformation $F=(f,g)$ can be uniquely
decomposed as $F=\tilde F\circ\Lambda$, where
\begin{equation}
\label{dilations}\Lambda(z,w)=(\lambda z,\lambda^3
w),\,\lambda\in\RR{}\setminus\{0\},
\end{equation}
is an automorphism of the model \eqref{model} and the weighted
components of the new mapping $\tilde F$ satisfy
\begin{equation}
\label{normalizedmap} f_1=z,\quad g_1=g_2=0,\quad g_3=w.
\end{equation}
Hence in what follows we can restrict ourself to transformations,
satisfying \eqref{normalizedmap}
After these preparations we consider \eqref{tangency} as an
infinite series of weighted homogeneous equations, which can be
written for any fixed weight $m\geq 4$ as
\begin{equation}
\label{CM}\re\left(ig_{m}+(2z\bar z+\bar
z^2)f_{m-2}\right)\left|_{w=u+i(z^2\bar z+z\bar
z^2)}\right.=\Phi^*_m-\Phi_m+\cdot\cdot\cdot,
\end{equation}
where dots stands for a polynomial in $z,\bar z,u$ and
 $f_{j-2},g_{j}$ with $j<m$ and their
derivatives in $u$ (here
$f_{j-2}=f_{j-2}(z,u),\,g_j=g_{j}(z,u),\,\Phi_j=\Phi_j(z,\bar
z,u))$.

Let us denote by $\mathcal F$ the space of formal real-valued
power series $\Phi(z,\bar z,u)=\sum_{m\geq 4}\Phi_m(z,\bar z,u)$,
satisfying \eqref{simplified}, by $\mathcal N\subset\mathcal F$
the subspace of series satisfying the normalization conditions
\eqref{normalform}, and by $\mathcal G$ the space of pairs $(f,g)$
of formal power series without constant term, satisfying
\eqref{normalizedmap}.
In view of \eqref{CM}, in order to prove the formal version of
Theorem 1, it is sufficient now to prove

\begin{propos}
For the linear operator $$ L(f,g):=\re\left(ig+(2z\bar z+\bar
z^2)f\right)\left|_{w=u+i(z^2\bar z+z\bar z^2)}\right.$$  we have
the direct sum decomposition $\mathcal F=L(\mathcal
G)\oplus\mathcal N$.
\end{propos}

Indeed, it follows from Proposition 2.1 that if $m\geq 4$ is an
integer and all $f_{j-2},g_{j},\Phi^*_j$ with $j<m$ are already
determined, then one can uniquely choose the collection
$(f_{m-2},g_{m},\Phi^*_m)$ in such a way that \eqref{CM} is
satisfied and $\Phi^*_m\in\mathcal N$. This implies the existence
and uniqueness of the desired normalized mapping $(f,g)$ and the
normalized right-hand side $\Phi^*(z,\bar z,u)$.

\begin{proof}[Proof of Proposition 2.1]
The statement of the proposition is equivalent to the fact that an
equation $L(f,g)=\Psi(z,\bar z,u),\,(f,g)\in\mathcal G$ in $(f,g)$
has a unique solution, modulo $\mathcal N$ in the right-hand side,
for any fixed $\Psi\in\mathcal F$. To simplify the calculations,
we replace an equation $L(f,g)=\Psi(z,\bar z,u)$ by the equation
\begin{equation}
\label{cm-eq} 2L(f,g)=\Psi(z,\bar z,u),\quad (f,g)\in\mathcal
G,\,\Psi\in\mathcal F,
\end{equation}
which we solve in $f,g$. We use expansions of the form
$$f(z,u+i(z^2\bar z+z\bar z^2))=f(z,u)+f_u(z,u)i(z^2\bar z+z\bar
z^2)+\frac{1}{2}f_{uu}(z,u)i^2(z^2\bar z+z\bar
z^2)^2+\cdot\cdot\cdot.$$ Substituting into \eqref{cm-eq} we get
the equation
\begin{multline}
\label{cm-eq1} i\left(g(z,u)+g_u(z,u)i(z^2\bar z+z\bar
z^2)+\frac{1}{2}g_{uu}(z,u)i^2(z^2\bar z+z\bar
z^2)^2+\cdot\cdot\cdot\right)+\\
 +(2z\bar z+\bar z^2)\left(f(z,u)+f_u(z,u)i(z^2\bar z+z\bar
z^2)+\frac{1}{2}f_{uu}(z,u)i^2(z^2\bar z+z\bar
z^2)^2+\cdot\cdot\cdot\right)+\\
+ \,\{\mbox{complex conjugate terms}\}=\Psi(z,\bar z,u).
\end{multline}
We expand $f(z,w)$ as $f=\sum\limits_{k\geq 0}f_k(w)z^k,$ and
similarly for $g$. Then, collecting in \eqref{cm-eq1} terms of
bi-degree $(k,0),\,k\geq 3$ in $z,\bar z$, we get
\begin{equation}
\label{k0} ig_k(u)=\Psi_{k0}(u).
\end{equation}
Gathering then terms of bi-degree $(k+1,1)$ with $k\geq 4$, we get
\begin{equation}
\label{k1}-g'_{k+1}(u)+2f_k(u)=\Psi_{k+1,0}(u).
\end{equation}
The equations \eqref{k0}, \eqref{k1} enable us to determine $g_k$
with $k\geq 3$ and $f_k$ with $k\geq 4$ uniquely. We then proceed
further with comparing terms of fixed bi-degree. Gathering terms
of bi-degrees $(1,0)$ and $(3,1)$ respectively, we get
\begin{eqnarray}
\label{1031} ig_1(u)=\Psi_{10}(u)\\
 -g'_1(u)+2f_2(u)=\Psi_{31}(u).\notag
 \end{eqnarray}
 The system
\eqref{1031} determines the pair $(g_1,f_2)$ uniquely. Further,
gathering $(1,1)$ terms, we get
$$4\re f_0(u)=\Phi_{11}(u),$$
so that only $\im f_0(u)$ needs to be determined. Gathering then
terms of bi-degrees $(2,0),\,(4,1)$ and $(3,2)$ respectively, and
separating real and imaginary parts, we obtain the system
\begin{eqnarray}
\label{204132}
 -\im g_2(u)=\re\Psi_{20}(u)\notag\\
  \re g_2(u)-\im
f_0(u)=\im\Psi_{20}(u)\notag\\
-\re g'_2(u)+2\re f_3(u)-\im f'_0(u)=\re\Psi_{41}(u)\\
 -\im g'_2(u)+2\im f_3(u)=\im\Psi_{41}(u)\notag\\
 -\re g'_2(u)+\re f_3(u)-3\im f'_0(u)=\re\Psi_{32}(u).\notag
\end{eqnarray}

The first and the fourth equations in \eqref{204132} determine the
unknowns $\im g_2,\im f_3$ uniquely. Differentiating the second
equation and considering it together with the third and the fifth
equations, we get a real linear system for $\re g'_2,\re f_3,\im
f'_0$ with a non-zero determinant, thus $\re g'_2,\re f_3,\im
f'_0$ are also determined uniquely. Since $f(0,0)=0$, we have
$f_0(0)=0$, and since $\Psi$ does not contain terms of weight $2$,
we have $\Psi_2(0)=0$, so that from the second equation we obtain
$\re g_2(0)=0$, and this determines  $\re g_2,\re f_3,\im f_0$
uniquely.

It remains to determine $g_0$ and $f_1$. Gathering terms of
bi-degrees $(0,0),(2,1)$ and $(4,2)$, and separating real and
imaginary parts, we obtain the system

\begin{eqnarray}
\label{002142}
 -2\im g_0(u)=\Psi_{00}(u)\notag\\
- 2\re g'_0(u)+3\re f_1(u)=\re\Psi_{21}(u)\\
 \im f_1(u)=\im\Psi_{21}(u)\notag\\
-\im g'_3(u)+\im g''_0(u)+\re f'_1(u)+\im
f_4(u)=\Im\Psi_{42}(u).\notag
\end{eqnarray}
 We then determine, respectively, $\im g_0$
from the first, $\im f_1$ from the third, $\re f'_1$ from the
fourth, and $\re g'_0$ from the second equations in
\eqref{002142}. Since we have $f_{1}(0)= g'_0(0)=1$, we determine
from here $\re f_1,\re g_0$ uniquely (where $f_{4}$ is already
determined from \eqref{k1}).

Thus the map $(f,g)$ is uniquely determined. Since the collection
of all $\re \Psi_{kl},\im \Psi_{kl}$ considered above must vanish
for $\Psi\in\mathcal N$, while all the remaining $\re
\Psi_{kl},\im \Psi_{kl}$ for for $\Psi\in\mathcal N$ can be
arbitrary, this proves the proposition.
\end{proof}

The formal version of Theorem 1  is proved now. Using the fact
that any transformation \eqref{dilations} preserves the
normalization conditions \eqref{normalform}, we can also formulate

\begin{corol}
[see also \cite{belsmall},\cite{kolar}] The group of formal
invertible transformations, preserving the germ at $0$ of the
model hypersurface \eqref{model}, consists of the dilations
\eqref{dilations}.
\end{corol}

\begin{corol}
If $(N,0)$ and $(\tilde N,0)$ are two different normal forms of a
fixed germ $(M,p)$ with generic Levi degeneracy at $p$, then there
exists a linear transformation $\Lambda$, as in \eqref{dilations},
which maps $(N,0)$ into $(\tilde N,0)$.
\end{corol}

\subsection{Convergence of the normalizing transformation}

\mbox{}\\

Our next goal is to prove that the constructed normalizing
transformation $F=(f,g)$, satisfying \eqref{normalizedmap}, is in
fact convergent. We do that by presenting $F$ as a composition of
certain holomorphic transformations. Each of the transformation
has a clear geometric interpretation, that we address below. For
the set-up of the theory of Segre varieties see, e.g., \cite{ber}.

\medskip

\noindent \bf Canonical pair of foliations in the Levi degeneracy
set. \rm Let $\Sigma$ be the Levi degeneracy set of a
real-analytic hypersurface $M\subset\CC{2}$ with generic Levi
degeneracy at the point $0\in \Sigma\subset M$. We first define
the following slope (line) field in the totally real submanifold
$\Sigma$. Choose a point $p\in \Sigma$ and coordinates $(z,w)$
vanishing at $p$, where $M$ takes the form \eqref{simplified}.
Clearly these coordinates can be chosen polynomial with
coefficients depending analytically on $p$.

Let $N$ denotes a (formal) normal form  \eqref{normalform} of $M$
at $p$, $F$ a formal transformation, mapping $(M,p)$ onto $(N,0)$,
and $e:=(0,1)\in\CC{2}$. We then define a slope at $p$ as follows:
$$l(p):=\mbox{span}_{\RR{}}\left\{(dF|_p)^{-1}(e)\right\}\subset T_{p}M.$$
It follows from Corollary 3.3 that the definition of $l$ is
independent of the choice of normal form. Moreover, the desired
slope can be also defined without using formal transformations.
Indeed, it follows from the normal form construction that, as soon
as the initial weighted polynomials
$\{\Phi^*_j,f_{j-2},g_j,\,4\leq j\leq m\}$ for some $m\geq 4$ have
been determined, they do not change after further normalization of
terms of higher weight. Hence, solving the equations \eqref{CM}
for $m$ sufficiently large, we uniquely determine
$dF|_p$. It is not difficult to see that the constructed slope
field is analytic (i.e., depends analytically on a point
$p\in\Sigma$). Indeed, the explicit construction in the beginning
of the section shows that each fixed weighted polynomial $\Phi_m$,
as in \eqref{simplified}, depends on $p$ analytically (this can be
verified from the parameter version of the implicit function
theorem). Hence polynomials $f_m$ and $g_m$ depends on $p$
analytically, as it is obtained by solving a system of linear
equations with fixed nondegenerate matrix in the left-hand and
right-hand side analytic in $p$ (the latter fact can be seen from
the proof of Proposition 2.1). We immediately conclude that
$dF|_p$ depends on $p$ analytically, and so does $l(p)$.

We then integrate the analytic slope field $l(p)$ and obtain a
canonical (non-singular) real-analytic foliation $\mathcal T$ of
the totally real manifold $\Sigma$ of Levi-degenerate points.

\begin{dfn}
We call each of the leaves of the foliation $\mathcal T$ \it a
degenerate chain. \rm
\end{dfn}

 Each degenerate chain
$\gamma\subset\Sigma$ at each point $p\in\gamma$ is transverse to
the complex tangent $T^{\CC{}}_pM$.

The second canonical foliation in $\Sigma$ corresponds to the
slope field $$c(p):=T^{\CC{}}_p M\cap T_p\Sigma.$$ Integrating
$c(p)$ we obtain another canonical foliation $\mathcal S$ in
$\Sigma$, which is everywhere tangent to $T^{\C}M$.
Both foliations $\mathcal T$ and $\mathcal S$ are transverse to
each other and are biholomorphic invariants of $(M,0)$. We call
them respectively {\em transverse} and {\em tangent} foliations.

\medskip

\noindent \bf Normalization of a chain and of the field of complex
tangent vectors along the chain. \rm We start the construction of
a convergent transformation, mapping a germ $(M,0)$, as in
\eqref{simplified}, onto a germ $(N,0)$, as in \eqref{normalform},
by choosing the unique degenerate chain $\gamma\subset\Sigma$,
passing through $0$. Let us denote by $s(p)$ the leaf of the
tangent foliation $\mathcal S$, passing through a point
$p\in\Sigma$. As $\Sigma$ is totally real in $\CC{2}$, we may
perform a local holomorphic coordinate change near $0$ in such a
way that the form \eqref{simplified} is preserved, the Levi
degeneracy set $\Sigma$ is transformed into
\begin{equation}
\label{flat} \Pi=\{\re z=0,\,\im w=0\},
\end{equation}
the degenerate chain $\gamma$ into
\begin{equation}
\label{Gamma}\Gamma=\{z=0,\,v=0\}\subset\Pi,
\end{equation}
and each leaf $s_p$, $p\in\gamma$ of the tangent foliation near
$0$ into $\{\re z=0,w=b\},\,b\in\RR{}$. Thus for the transformed
hypersurface $M^*$, all complex tangent spaces
$T_p^{\CC{}}M,\,p\in\Gamma$, are of the form $\{w=0\}$. Since
$\Sigma$ is a plane for $M^*$, we have
$$\Phi^*_{11}(u)=0,$$
expressing the fact that the Levi form of $M$ vanishes along
$\Gamma$. Furthermore, since $M^*$  contains $\Gamma$, we also
obtain the condition $\Phi^*_{00}(u)=0$. In what follows we
consider only transformations, preserving $\Gamma$ (we will prove
later that when $M$ is already in the normal form, then the curve
$\Gamma$ coincides with the degenerate chain, passing through
$0$). Note that in the normal form \eqref{normalform} the Levi
degeneracy set $\Sigma$ \it is not necessarily flat, \rm as in
\eqref{flat}, and neither are the leaves $s_p,\,p\in\Gamma$.
However, all the complex tangents $T_p^{\CC{}}M,\,p\in\Gamma$,
remain $\{w=0\}$, and the condition $\Phi^*_{11}(u)=0$ is also
preserved.

\medskip

\noindent \bf Normalization of Segre varieties along a chain. \rm
The next step in the normalization procedure is the elimination of
$(k,0)$ terms in the expansion of $\Phi$, which is sometimes
addressed as transfer to \it normal coordinates \rm (see
\cite{ber}). Geometrically, this step can be interpreted as
straightening of the Segre varieties $Q_p,\,p\in\Gamma$. According
to \cite{chern}, \cite{ber}, we perform the unique transformation
of the form
\begin{equation}
\label{kill-k0} z^*=z,\quad w^*=w+g(z,w), \quad g(0,w)=0,
\end{equation}
which maps $M$ into a hypersurface with $\Phi^*_{k0}=0,\,k\geq 0$.
This transformation preserves the curve $\Gamma$. One needs only
to take control of the fact that the conditions $\Phi_{11}(u)=0$
and $\Phi_{21}(0)=1$ are preserved. Indeed, one clearly has
$\Phi^*_{11}(u^*)= 0$, since the Levi form of $M^{*}$ vanishes
along $\Gamma$.
The substitution $z^*=z,\,w^*=u+i\Phi(z,\bar z,u)+O(|z|)$ into
$v^*=\Phi^*(z^*,\bar z^*,u^*)$ also shows $\Phi^*_{21}(0)=1$.

\medskip

\noindent \bf Normalization of the Segre map. \rm This step can be
interpreted as a normalization of the Segre map. The latter one,
since Segre varieties are determined by their $2$-jets, can be
regarded as a map $p\mapsto j^2Q_p|_{z=0}$ of $(\CC{2},0)$ into
the $2$-jet space of Segre varieties at the point with $z=0$. We
normalize the Segre map in this step by the condition
\begin{equation}
\label{Segre}p=(\xi,\eta)\mapsto
\left(\bar\eta,2i\bar\xi^2,4i\bar\xi+O(\bar\xi^2)\right).
\end{equation}

Let us introduce the subspace $\mathcal D\subset\mathcal F$, which
consists of all convergent power series of the form
$$\sum\limits_{k,l\geq 2}\Psi_{kl}(u)z^k\bar z^l.
$$
Our goal is to bring a hypersurface $M$, obtained in the previous
step, to such a form that all terms of weight $\geq 4$ in $\Phi$
belong to the space $\mathcal D$, i.e., to bring the defining
equation to the form
\begin{equation}
\label{prenormal} v=P(z,\bar z) \mod \mathcal D.
\end{equation}
Thus the subspace of terms to be removed from $\Phi$ in this step
is transverse to $\mathcal D$, and adding to $\Phi$ an element of
$\mathcal D$ does not change the desired form of it, that is why
it is convenient for us to use identities modulo $\mathcal D$.

Consider a hypersurface $M$, obtained in the previous step. We
first make $\Phi_{21}$ independent of $u$. Let $\lambda(u)$ be an
analytic function with $\lambda(0)=1$ such that
$\lambda^2(u)\bar\lambda(u)=\Phi_{21}(u)$. We perform the
biholomorphic change
$$z^*=z\lambda(w),\,w^*=w.$$
Using expansions of the form $h(u+iv)=h(u)+ih'(u)v+...$, we
compute
$$
\Phi(z,\bar z,u)=\Phi^*(z^*,\bar z^*,u^*)=\Phi^*(z\lambda(u),\bar
z\bar\lambda(u),u)\quad\mbox{mod}\,\mathcal D,
$$
provided $(z,w)\in M$ (recall that all $(k,0)$-terms are removed
from $\Phi$). Thus
$$\Phi_{21}(u)=\Phi^*_{21}(u)\lambda^2(u)\bar\lambda(u),$$
so that $M$ is mapped into a hypersurface of the form
$$
v=2\re\left(z^2\bar z\right)+\sum\limits_{k,l\geq 1,\,k+l\geq
4}\Phi^*_{kl}(u)z^k\bar z^l, $$ as required. Clearly, $\Gamma$ is
preserved and $\Phi^*_{11}=0$.

 Second, for the hypersurface $M$, obtained in the
previous step, we remove all $(k,1)$-terms in $\Phi$ with $k\geq
3$ by a transformation
$$z^*=z+f(z,w),\quad w^*=w, \quad f(z,w)=O(|z|^2).$$
We have
\begin{equation}
\label{map} \Phi^{*}(z+f(z,u+i\Phi(z,\bar z,u)), \bar z +
\overline{f(z,u+i\Phi(z,\bar z,u))}, u) = \Phi(z,\bar z, u).
\end{equation}
Furthermore $\Phi^{*}_{21}(u)=\Phi_{21}(u)$ and
$\Phi^{*}_{11}(u)=\Phi_{11}(u)=0$.

We expand the defining function $\Phi$ as
$$\Phi(z,\bar
z,u)=2\re\left(z^2\bar z+z^2\phi(z,u)\bar
z\right)\quad\mbox{mod}\,\mathcal D$$
 for
an appropriate $\phi(z,u)=O(|z|)$.  Put $f=:zh$. We compute
$$(z+f)^2(\bar z+\bar f)=z^2\bar
z(1+2h+h^2)\quad\mbox{mod}\,\mathcal D.$$ We then determine
$h(z,u),f(z,u)$ from the functional equation
\begin{equation}
\label{f21}
 2h+h^2= \phi(z,u),
\end{equation}
expressing the condition that $M$ is mapped into a hypersurface
$M^{*}$ given by
 \eqref{prenormal}.
Now suitable $h$ can be obtained by the implicit function theorem.
 Note that the required transformation,
removing the $(k,1)$ terms, is unique.
  It is straightforward that the Segre map is
given by \eqref{Segre}.

\medskip

\noindent \bf  Fixing a parametrization for a chain. \rm We claim
now that $\re\Phi_{32}(u)=0$ in \eqref{prenormal}. Indeed,
consider the (formal) transformation $F=(f,g)$, bringing a
hypersurface \eqref{prenormal} into normal form, and study the
equation \eqref{tangency}, applied to it. Collecting terms with
$z\bar zu^l,\,l\geq 0$, we first obtain $\re f_0(u)=0$. Next,
gathering terms with $z^2\bar z^0u^1,\,z^4\bar z^1u^0,\,z^3\bar
z^2u^0$ and separating the real and imaginary parts, it is
straightforward to check that we obtain precisely the equations
\eqref{204132} (and also the differentiated second equation in
\eqref{204132}), evaluated at $u=0$, with
$\Psi_{20}(0)=\Psi'_{20}(0)=\Psi_{41}(0)=0$ and
$\re\Psi_{32}(0)=2\re\Phi_{32}(0)$ (this follows from the partial
normalization \eqref{prenormal} of $M$ and the normalization
conditions \eqref{normalform} for the target $M^*$). Moreover,
$\im f_0(0)=\im f'_0(0)=0$ thanks to the fact that $\Gamma$ is a
degenerate chain. We immediately conclude from \eqref{204132} that
$\re g_2(0)=\re g'_2(0)=\im g'_2(0)=\re f_3(0)=\im f_3(0)=0$ and
$\re\Phi_{32}(0)=0$. Since the prenormal form \eqref{prenormal} is
invariant under the real shifts $w\mapsto w+u_0,\,u_0\in\RR{}$,
and $\Gamma\subset M$ is a degenerate chain, we similarly conclude
that $\re\Phi_{32}(u_0)= 0$ for any small $u_0\in\RR{}$ in
\eqref{prenormal}, as required.

It remains to achieve the last normalization condition $\im
\Phi_{42}(u)=0$, using an appropriate \it gauge \rm transformation
\begin{equation}
\label{gauge} z\mapsto f(w)z,\quad w\mapsto g(w),\quad f(0)\neq
0,\,g'(0)\neq 0.
\end{equation}
We do so by choosing a gauge transformation
$$z^*=z(q'(w))^{1/3},\,w^*=q(w)$$ for an appropriate
$q(w)$ with $\im q(u)=0,\,q(0)=0,q'(0)\neq0$, which can be
interpreted as a choice of parametrization for the degenerate
chain $\gamma$, determined in the previous step. We apply the
above transformation to a hypersurface, satisfying
\eqref{prenormal} and $\re\Phi_{32}(u)=0$. Recall that $\mathcal
N$ denotes the space of power series in $z,\bar z,u$ of weight
$\geq 4$, satisfying the normalization conditions
\eqref{normalform}. Then, using expansions of kind
$h(u+iv)=h(u)+ih'(u)v+\dots,$ we compute:
\begin{gather*}
v^*=q'(u)v\quad\mbox{mod}\,\mathcal N,\\
{z^*}^2\bar z^*=z^2\bar z
q'(u)\left(1+i\frac{q''(u)}{q'(u)}v)\right)^{2/3}\left(1-i\frac{q''(u)}{q'(u)}v)\right)^{1/3}\quad\mbox{mod}\,\mathcal
N =\\
=q'(u)z^2\bar z+\frac{i}{3}q''(u)z^4\bar
z^2\quad\mbox{mod}\,\mathcal N,\\
v^*-2\re({z^*}^2\bar z^*)=q'(u)(v-2\re(z^2\bar
z))+\frac{i}{3}q''(u)(z^4\bar z^2-z^2\bar
z^4)\quad\mbox{mod}\,\mathcal N.
\end{gather*}
We conclude that $\im \Phi^*_{42}=q'\im \Phi_{42}+\frac{1}{3}q''$,
so that the condition $\im \Phi^*_{42}(u)=0$ leads to a second
order \it nonsingular \rm ODE. Solving it with some initial
condition $q'(0)\neq 0$, we finally obtain a hypersurface of the
form \eqref{prenormal}, satisfying $\re \Phi_{32}(u)=\im
\Phi_{42}(u)=0$, as required for the complete normal form. It is
not difficult to see, performing similar calculations, that the
gauge transformation chosen to achieve $\im \Phi_{42}=0$ must have
the above form $z^*=z(q'(w))^{1/3},\,w^*=q(w)$ and hence is unique
up to the choice of the real parameter $q'(0)$, corresponding to
the action of the group \eqref{dilations}. Thus, remarkably,\it

\smallskip

we can canonically, up to the action of the group of dilations
\eqref{dilations}, choose a parametrization on each degenerate
chain.\rm
\medskip

Theorem 1 is completely proved now. The proof shows also that in
the normal form \eqref{normalform} the unique degenerate chain,
passing through the origin, is given by \eqref{Gamma}. In
addition, we can see that

\smallskip

\it a transformation, bringing a germ of a real-analytic
hypersurface $M\subset\CC{2}$ with generic Levi degeneracy at $0$
into normal form, is completely determined by fixing a tangent
vector at $0$ to the degenerate chain, passing through $0$. \rm

\smallskip

\section{The higher dimensional case}

Let $M\subset\CC{n+1},\,n\geq 2$ be a real-analytic hypersurface
with generic Levi degeneracy at $0$. Choosing local coordinates
for $M$ as discussed in Section 2, and performing an additional
polynomial transformation removing harmonic terms of degrees $2$
and $3$, we can present $M$ locally near the origin as

 \begin{equation}
\label{initialn} v =\left\langle \check z,\overline{\check
z}\right\rangle +2\re\left( z_n^2\bar z_n\right) +O\bigl(|\check
z|^2|z_n|+|z|^4+u|z|+u^2\bigr).
\end{equation}
(compare with the initial simplification in \cite{ebenfeltjdg}).

\subsection{Formal normalization in the higher dimensional case}

A crucial step in the construction of a normal form is again a
good choice of weights for a hypersurface given by
\eqref{initialn}. We introduce weights as follows:

$$[\check z]:=3,\quad [z_n]:=2,\quad [w]:=6.$$ Then the hypersurface
\begin{equation}
\label{modeln} v=P(z,\bar z),
\end{equation}
where
$$P(z,\bar z):=\left\langle \check z,\overline{\check
z}\right\rangle+2\re\left(z_n^2\bar
z_n\right)=\sum\limits_{j=1}^{n-1}\varepsilon^j\bigl|z_j\bigr|^2+2\re\left(
z_n^2\bar z_n \right)$$ is a weighted homogeneous polynomial of
weight $6$, becomes a "model"\, for the class of hypersurfaces
\eqref{initialn} that can be written as
\begin{equation}
\label{simplifiedn} v=P(z,\bar z)+\sum_{m\geq 7}\Phi_m(z,\bar
z,u),
\end{equation}
where each $\Phi_m(z,\bar z,u)$ is homogeneous of weight $m$.

Let now $M=\{v=\Phi(z,\bar z,u)\}$ and $M^*=\{v^*=\Phi^*(z^*,\bar
z^*,u^*)\}$ be two hypersurfaces in $\CC{n+1}$, satisfying
\eqref{initialn}, and $z^*=f(z,w),\,w^*=g(z,w)$  a formal
invertible transformation, transforming $(M,0)$ into $(M^*,0)$. We
have the identity
\begin{equation}
\label{tangencyn} \im g(z,w)|_{w=u+i\Phi(z,\bar z,u)}
=\Phi^*(f(z,w),\overline{f(z,w)},\re g(z,w))|_{w=u+i\Phi(z,\bar
z,u)}.
\end{equation}
For a mapping $(f,g)=(f^1,...,f^n,g):\, (\CC{n+1},0)\mapsto
(\CC{n+1},0)$ we group the first $n-1$ components in the vector
function
$$\check
f=(f^1,...,f^{n-1}).$$ We decompose the mapping into a sum of
weighted homogeneous polynomials as
$$f=\sum_{m\geq 2} f_m(z,w),\quad \check f=\sum_{m\geq 2} \check f_m(z,w),\quad g=\sum_{m\geq 2} g_m(z,w)$$ and consider terms
of a fixed weight in \eqref{tangencyn}. Collecting terms of the
low weights $2,...,5$, we get in view of \eqref{initialn},
$$\check f_2=g_2=...=g_5=0.$$
 For the
weight $6$ we have

$$\check f_3=\lambda^3 C\check z,\, f^n_2=\rho\lambda^2 z_n,
\,g_6=\rho\lambda^6 w,$$ where $$\left\langle C\check
z,\overline{C\check z}\right\rangle=\rho\left\langle \check
z,\overline{\check z}\right\rangle,\quad \lambda>0,\quad
C\in\mbox{\sf GL}(n-1,\CC{}),\quad\rho\in\{1,-1\}$$ (in fact, for
$r=n-1$ we have $\rho=1$ only).  Thus the initial transformation
$F=(f,g)$ can be uniquely decomposed as $F=\tilde F\circ\Lambda$,
where
\begin{eqnarray}
\label{dilationsn} \Lambda(z,w)&=&(\lambda^3C\check
z,\rho\lambda^2
z_n,\rho\lambda^6 w), \notag\\
\left\langle C\check z,\overline{C\check
z}\right\rangle=\rho\left\langle \check z,\overline{\check
z}\right\rangle,\quad C&\in&\mbox{\sf
GL}(n-1,\CC{}),\quad\lambda>0,\quad\rho\in\{1,-1\}
\end{eqnarray}
 is
an automorphism of the model \eqref{modeln}, and the weighted
components of $\tilde F$ satisfy
\begin{equation}
\label{normalmap} \check f_2 = 0, \quad \check f_3=\check z,\quad
f^n_2=z_n, \quad \quad g_2=...=g_5=0,\quad g_6=w.
\end{equation}
Thus in what follows we only consider maps, satisfying
\eqref{normalmap}. Proceeding further with weighted identities
arising from \eqref{tangencyn}, and using \eqref{simplifiedn}, we
obtain, similarly to the two-dimensional case
\begin{equation}
\label{CMn}\re\Bigl(ig_{m}+\langle \check f_{m-3},\overline{\check
z}\rangle+\left(2z_n\overline{z_n}+(\overline{z_n})^2\Bigr)f^n_{m-4}\right)\left|_{w=u+iP(z,\bar
z)}\right.=\Phi^*_m-\Phi_m+\cdot\cdot\cdot,
\end{equation}
for any fixed weigh $m\geq 7$, where dots stand for a polynomial
in $z,\bar z,u$ and
 $\check f_{j-3}(z,w),f^n_{j-4}(z,w),g_j(z,w)$ with $j<m$,
 with $w=u+iP(z,\bar z)$.

We now decompose the space $\mathcal F^n$ of formal real-valued
power series
$$\Phi(z,\bar z,u)=\sum_{m\geq 7}\Phi_m(z,\bar z,u),$$
satisfying \eqref{initialn}, into  the direct some of the image
$\mathcal R^n$ of the operator on left-hand side \eqref{CMn},
called the {\em Chern-Moser operator} associated to $P$:
$$
L^n(f,g):= \re(ig + P_zf)|_{w=u+iP(z,\bar z)}
 =  \re\Bigl(ig+\langle \check f,\overline{\check
z}\rangle+\left(2z_n\overline{z_n}+(\overline{z_n})^2\Bigr)f^n\right)\left
|_{w=u+iP(z,\bar z)}\right. ,$$ acting on formal power series
$$(f,g)\in \C[[z,w]]^n \times \C[[z,w]],$$
and an appropriate normal subspace $\mathcal N^n\subset\mathcal
F^n$. Let also $\mathcal G^n$ denotes the space  of all
$(n+1)$-tuples $(f,g)$ of power series without constant terms,
satisfying \eqref{normalmap}. As was already discussed in Section
3, in order to prove the formal version of Theorem 2,
it remains to prove:

\begin{propos}
For the  Chern-Moser operator $L^n$,
 we have the direct decomposition
 $$\mathcal
F^n=L^n(\mathcal G^n)\oplus\mathcal N^n,$$ where $\mathcal
N^n\subset\mathcal F^n$ is the subspace of series, satisfying the
normalization conditions \eqref{normalformn}.
\end{propos}

\begin{proof}
We have to prove that an equation
\begin{equation}
\label{cm-eq-n} 2L^n(f,g)=\Psi(z,\bar z,u), \quad (f,g)\in\mathcal
G^n,
\end{equation}
has a unique solution in $(f,g)$, modulo $\mathcal N^n$ in the
right-hand side, for any fixed $\Psi\in\mathcal F^n$.  We expand
\eqref{cm-eq-n} as
\begin{multline}
\label{cmeqn} i\left(g(z,u)+g_u(z,u)i P(z,\bar
z)+\frac{1}{2}g_{uu}(z,u)i^2(P(z,\bar z))^2+\cdot\cdot\cdot\right)\\
 +\Bigl\langle \check f(z,u)+\check f_u(z,u)iP(z,\bar z)+\dots,\overline{\check z}\Bigr\rangle+
 (2z_n\overline{z_n}+(\overline{z_n})^2)\Bigl(f^n(z,u)+f^n_u(z,u)iP(z,\bar z)+\cdot\cdot\cdot\Bigr)+\\
 \qquad+ \{\mbox{complex conjugate terms}\}=\Psi(z,\bar
 z,u),
\end{multline}
and each $f^s(z,w),\,1\leq s\leq n$, as
$$f^s=\sum f^s_{\a l}(w)\check
z^{\a}z_n^l, \quad \check z^\a = z_1^{\a_1} \cdots
z_{n-1}^{\a_{n-1}}, \quad
\a=(\a_1,\ldots,\a_{n-1})\in\left(\mathbb{Z}_{\geq
0}\right)^{n-1},\quad l\in\mathbb{Z}_{\geq 0},$$ and similarly for
$g$. We expand the right-hand side as
\begin{gather*}
\Psi(z,\bar z,u)
=\sum \Psi_{k'kl'l}(\check z,\1{\check z},u)z_n^k\bar z_n^l, \quad 
k',k,l',l\in\mathbb{Z}_{\geq 0},
\end{gather*}
where each $\Psi_{k'kl'l}$ is bi-homogeneous of bi-degrees
$(k',l')$ in $(\check z,\overline{\check z})$ respectively. We are
looking for the normal space $\6N^n$ of the form
$$\6N^n=\bigoplus\6N_{k'kl'l},
\quad \6N_{k'kl'l}\subset \bigl\{\Psi(z,\bar z,u) =
\Psi_{k'kl'l}(\check z,\1{\check z},u)z_n^k\bar z_n^l\bigr\}.$$

We start by considering certain ``mixed"\, terms in the basic
equation \eqref{cmeqn} in order to simplify further calculations.
First, we collect terms of the form $\CC{}[[u]]z_n\bar z_n$ in
\eqref{cmeqn} and get
\begin{equation}\label{fn00}
4\re f^n_{00}(u)=\Psi_{0101}(u).
\end{equation}
Collecting then in \eqref{cmeqn} all terms of the form
$\CC{}[[u]]\check z^\a \overline{\check z}^\b z_n\bar
z_n,\,|\a|=|\b|=1$, and applying the $\mbox{\bf tr}$ operator, we
obtain
\begin{equation}
\label{findfn00}-4\im (f^n_{00})'(u)=\frac{1}{n-1}\mbox{\bf
tr}\,\Psi_{1111}.
\end{equation}
Thus $\im (f^n_{00})'(u)$ is determined uniquely. In view of
$f^n_{00}(0)=0$, there exists unique $f^n_{00}(u)$ satisfying
\eqref{fn00} and \eqref{findfn00}.
Consequently we can set
\begin{equation}
\6N_{0101} = 0, \quad \6N_{1111} = \ker\tr,
\end{equation}
for the components of the normal space $\mathcal N^n$.

Second, we collect in \eqref{cmeqn} all the terms $\CC{}[u]\check
z^\a,\,|\a|=1$, as well as terms of the form $\CC{}[u]\check z^\a
z_n\bar z_n^2,\,|\a|=1$, and obtain the system
\begin{eqnarray}
\label{starting}
 i\sum_{|\a|=1}g_{\a0}(u)
\check z^\a+\langle \check z,\overline{\check
f_{00}}(u)\rangle=\Psi_{1000},\\
-\sum_{|\a|=1}g'_{\a0}(u) \check z^\a-i\langle \check
z,\overline{\check f'_{00}}(u)\rangle =\Psi_{1102}.\notag
\end{eqnarray}
As the form $\langle\cdot,\cdot\rangle$ is nondegenerate,
\eqref{starting} enables us to determine $\check
f'_{00}(u),g'_{\a0}(u),\,|\a|=1$, uniquely. Since we have $\check
f_{00}(0)=0$, and also $g_{\a0}(0)=0$ from \eqref{normalmap},
$|\a|=1$, equations \eqref{starting} completely determine $\check
f_{00}(u),g_{\a0}(u),\,|\a|=1$. Thus we can set
\begin{equation}
\6N_{1000}=\6N_{1102}=0.
\end{equation}

Having some of the initial terms of the mapping $(f,g)$
determined, we proceed further with solving the equation
\eqref{cmeqn}. Collecting all terms of the form $\CC{}[[z,u]]$, we
get
\begin{equation}
\label{hol} ig(z,u)-ig_{00}(u)+\langle \check z,\overline{\check
f_{00}}(u)\rangle+z_n^2\overline{f^n_{00}}(u)=\Psi(z,0,u).
\end{equation}
Since $f^n_{00}(u)$ is already determined, the equation
\eqref{hol} enables us to determine all $g_{\a l}(u)$ uniquely,
except  $\re g_{00}$ (while $\im g_{00}$ is also uniquely
determined). Keeping the latter conclusion in mind, we collect in
\eqref{cmeqn} all terms of the form $\CC{}[[z,u]]\overline{\check
z}^\a,\,|\a|=1$, to obtain
\begin{multline}
\label{holbar1} -g_u(z,u)\langle\check z,\overline{\check
z}\rangle +\langle \check f(z,u), \overline{\check z}\rangle
-i\sum_{|\a|=1}\overline{g_{\a0}}(u)\overline{\check z}^\a
-\overline{g'_{00}}(u)\langle\check z,\overline{\check z}\rangle
-i\sum_{|\a|=1}\overline{g_{\a0}}(u)
\overline{\check z}^\a+\\
+\sum_{|\a|=1}\langle\check z,\overline{\check f_{\a0}}(u)
\overline{\check z}^\a \rangle-i\langle \check z,\overline{\check
f_{00}}(u)\rangle\langle \check z,\overline{\check z}\rangle
+z_n^2\sum_{|\a|=1}\overline{f^n_{\a0}}(u) \overline{\check z}^\a
-i z_n^2\overline{(f^n_{00})'}(u)\langle \check z,\overline{\check
z}\rangle=\sum_{k',k\geq 0}\Psi_{k'k10}z_n^k,
\end{multline}
and also collect all terms of the form $\CC{}[[z,u]]z_n\bar z_n$
to obtain
\begin{multline}
\label{holbar2} -g_u(z,u)z_n^2\bar z_n +2z_n\overline{z_n}f^n(z,u)
-i\overline{g_{01}}(u)\overline{z_n}
-\overline{g'_{00}}(u)z_n^2\bar z_n
-i\langle \check z,\overline{\check f'_{00}}(u)\rangle z_n^2
\bar z_n\\
+2\overline{f^n_{00}}(u)z_n\overline{z_n}
+\overline{f^n_{01}}(u)z_n^2\bar z_n -i\overline{(f^n_{00}})'(u)
z_n^4\bar z_n =\sum_{k'\geq 0, k\ge 1}\Psi_{k'k01}z_n^k\bar z_n.
\end{multline}
Equations \eqref{holbar1} and \eqref{holbar2} enable us to
determine all $f_{\a l}(u)$ uniquely, except $\check f_{\a0}$ for
$|\a|=1$,  $f^n_{01}$ and $\check f_{02}$. To determine
$f^n_{\a0}$ with $|\a|=1$ and $\check f_{02}$ we first determine
$f^n_{\a0}$ by considering $\check z^\a z_n\overline {z_n}$ terms
with $|\a|=1$ in \eqref{holbar2}, and then $\check f_{02}$ by
considering $z_n^2\overline{\check z}^\a$  terms with $|\a|=1$ in
\eqref{holbar1}. Thus we have determined uniquely all $f^n_{\a k}$
for all $\a, k$ and $\check f_{\a k}$ for $|\a|\ne1$, and can set
$$\6N_{k'k10}=\6N_{k'k01}=0 \quad \text{ except }
\6N_{1010}, \; \6N_{0201}, \; \6N_{k'001}.$$

To determine the remaining terms in $f,g$, we consider in the
equation \eqref{cmeqn} terms of the form $\CC{}[[u]] \check z^\a
\overline{\check z}^\b z_n^2\bar z_n,\,|\a|=|\b|=1$. This gives
\begin{multline}
\label{complicated}
2\im g''_{00}(u)\langle \check z,\overline{\check z}\rangle
+i\sum\nolimits_{|\a|=1}\langle\check f'_{\a0}(u) \check
z^\a,\overline{\check z}\rangle -
\\
-i\sum\nolimits_{|\b|=1}\langle\check z, \overline{\check
f'_{\b0}}(u)\overline{\check z}^\b \rangle+
2i(f^n_{01})'(u)\langle \check z,\overline{\check z}\rangle-
i(\overline{f^n_{01}})'(u)\langle \check z,\overline{\check
z}\rangle
=\Psi_{1211}.
\end{multline}
For terms in $\CC{}[[u]] \check z^\a \overline{\check
z}^\b,\,|\a|=|\b|=1$, we obtain from \eqref{holbar1}
\begin{equation}
\label{k0alpha0} -2\re g'_{00}(u)\langle\check z,\overline{\check
z}\rangle +\sum_{|\a|=1}\langle\check f_{\a0}(u) \check z^\a,
\overline{\check z}\rangle +\sum_{|\b|=1}\langle\check
z,\overline{\check f_{\b0}}(u) \overline{\check
z}^\b\rangle=\Psi_{1010}.
\end{equation}
Finally, for terms in $\CC{}[[u]]z_n^2\bar z_n$ we obtain from
\eqref{holbar2}
\begin{equation}
\label{0201} -2\re g'_{00}(u)+2f^n_{01}(u)+\overline{f^n_{01}}(u)
=\Psi_{0201}.
\end{equation}
For a multiindex $\alpha\in\left(\mathbb{Z}_{\geq 0}\right)^n$
with $|\alpha|=1$, we use the notation $f^\alpha_{\gamma 0}$ for
the unique component $f^j_{\gamma0},\,1\le j\le n-1$, such that
$\alpha_j=1$. We first take the real part of \eqref{complicated},
 and then consider terms in
$\CC{}[[u]]\check z^\a \overline{\check z}^\b$ with
$|\a|=|\b|=1,\,\alpha\neq \b$; second, we consider in
\eqref{k0alpha0} the terms in $\CC{}[[u]] \check z^\a
\overline{\check z}^\b$ with $|\a|=|\b|=1,\,\a\neq \b$. We get the
system
\begin{equation}
\label{system1} i(f^\b_{\a0})'(u) -i\overline{(f^\a_{\b0})'}(u)
=\bigl(\re\Psi_{1211}\bigr)_{\check z^\a \overline{\check z}^\b},
\quad
f^\b_{\a0}(u)+\overline{f^\a_{\b0}}(u)=\left(\Psi_{\a0\b0}\right)_{\check
z^\a\overline{\check z}^\b}.
\end{equation}
Note that $\re\Psi_{1211}$ is already a Hermitian form.
Differentiating the second equation in \eqref{system1}, we obtain
for each $\a,\b$ as above a nondegenerate linear system for
$(f^\b_{\a0})'(u),\overline{(f^\a_{\b0})'}(u)$, that we can solve
uniquely. The compatibility of the solutions follows from the
reality of $\Psi$. From the normalization conditions
\eqref{normalmap} for $(f,g)$ we get $f^\b_{\alpha0}(0)=0$ for
$\alpha\neq \b$, so that all $f^\b_{\alpha0}(0)$ with $\alpha\neq
\b$ can be determined uniquely. Next, applying $\mbox{\bf tr}$ to
\eqref{complicated} and considering the imaginary part, we get

$$\im\left(2i(f^n_{01})'(u)-i\overline{(f^n_{01})'}(u)\right)=\frac{1}{2(n-1)}\im\Bigl\{\mbox{\bf tr}\,\Psi_{1211}(u)\Bigr\}.$$

Differentiating \eqref{0201} and considering the imaginary part,
we obtain

$$\im\left(2(f^n_{01})'(u)+\overline{(f^n_{01})'}(u)\right)=\im\Psi'_{0201}(u).$$
The two latter equations enable us to determine $(f^n_{01})'(u)$
uniquely by setting
$$\6N_{0201}=\ker\im, \quad \6N_{1211}=\ker\im\tr.$$

From \eqref{normalmap} we have $f^n_{01}(0)=1$, so that
$f^n_{01}(u)$ is completely determined. After that $\re
g'_{00}(u)$ is uniquely determined by considering the real part of
$\eqref{0201}$, and setting (together with previous normalization)
$$\6N_{0201}=0.$$
and this uniquely determines $\re g_{00}(u)$ thanks to the
condition $g_{00}(0)=0$ (recall that $\im g_{00}(u)$ was already
determined above). Finally, we determine $\re f^\a_{\a0}$ by
considering terms in $\CC{}[[u]]\check z^\a \overline{\check
z}^\a$ in the real part of \eqref{k0alpha0}, and determine $\im
(f^\a_{\a0})'$ by taking the real part
 and considering terms in
$\CC{}[[u]] \check z^\a \overline{\check z}^\a$ in
\eqref{complicated}. Together with previous normalization, this
amounts to setting
$$\6N_{1010}=0, \quad \6N_{1211}=\ker\im\tr \cap \ker\re.$$
Thanks to the conditions $f^\a_{\a0}(0)=1$ for $|\a|=1$, this
enables us to determine uniquely the entire mapping $(f,g)$, as
well as the right-hand side $\Psi$ modulo the subspace $\mathcal
N^n\subset\mathcal F^n$ of series, satisfying \eqref{normalformn}.
The proposition is proved now.
\end{proof}

Using, as in the one-dimensional case, the fact that any
transformation \eqref{dilationsn} preserves the normalization
conditions \eqref{normalformn}, we have:

\begin{corol}
The group of formal invertible transformations, preserving the
germ at $0$ of the model hypersurface \eqref{modeln}, consists of
the linear transformations \eqref{dilationsn}.
\end{corol}

\begin{corol}
If $(N,0)$ and $\tilde N,0)$ are two different normal forms of a
fixed germ $(M,p)$, where $M\subset\CC{n+1}$ is a real-analytic
hypersurface with generic Levi degeneracy at $p$, then there
exists a linear transformation $\Lambda$, as in
\eqref{dilationsn}, which maps $(N,0)$ into $(\tilde N,0)$.
\end{corol}

\subsection{Convergence of the normalizing transformation}

\mbox{}\\

It remains to prove that the transformation, satisfying
\eqref{normalmap} and bringing $(M,0)$ into a normal form, is a
composition of certain (convergent) biholomorphic transformations.
We describe these transformations below, giving a geometric
interpretation for each of them.

\medskip

\noindent \bf Canonical pair of foliations on the Levi degeneracy
set. \rm  We first define a pair of canonical foliations on the
Levi degeneracy set $\Sigma\subset M$. Recall that $\Sigma\subset
M$ is a codimension one submanifold, such that $T_p\Sigma\subset
T_p M$ is transverse to the Levi kernel $K_p\subset T^{\CC{}}_pM$
at every point $p\in\Sigma$. For a fixed point $p\in \Sigma$,
let $N$ denotes a (formal) normal form  \eqref{normalformn} of $M$
at $p$, $F$ a formal transformation, mapping $(M,p)$ onto $(N,0)$,
and $e^{n+1}$ the vector $(0,...,0,1)\in\CC{n+1}$. We then define
a slope at $p$ as follows:

$$l(p):=\mbox{span}_{\RR{}}\left\{(dF|_p)^{-1}(e^{n+1})\right\}.$$

It follows from Corollary 4.3 that the definition of $l(p)$ is
independent of the choice of a normal form and the corresponding
normalizing transformation. Arguing identically to the
two-dimensional case, we conclude that $l(p)$ depends on $p$
analytically. We integrate $l(p)$ and obtain a (non-singular)
real-analytic foliation $\mathcal T$ of the real manifold
$\Sigma$.

\begin{dfn}
We call each of the leaves of the foliation $\mathcal T$ \it a
degenerate chain. \rm
\end{dfn}
Each degenerate chain $\gamma\subset\Sigma$ at each point
$p\in\gamma$ is transverse to the complex tangent $T^{\CC{}}_pM$.

\medskip
The second canonical foliation in $\Sigma$ corresponds to the
slope field
$$c(p):=K_p\cap T_p\Sigma, \quad p\in\Sigma.$$
Integrating $c(p)$ we obtain a real-analytic foliation $\mathcal
S$ of $\Sigma$.
The canonical foliations $\mathcal T$ and $\mathcal S$, called \it
the transverse \rm and \it the tangent \rm foliations
respectively, are transverse to each other and are biholomorphic
invariants of $(M,0)$.

\medskip

\noindent \bf Normalization of a chain and of the Levi kernels
along the chain. \rm For a germ $(M,0)$ of a real-analytic
hypersurface $M\subset\CC{n+1}$, satisfying \eqref{initialn}, with
Levi degeneracy set $\Sigma$, we choose the unique degenerate
chain $\gamma\subset\Sigma$, passing through $0$. Let $s_p$
denotes a leaf of the foliation $\mathcal S$, passing though a
point $p\in\Sigma$. Consider then the set
$$S:=\bigcup_{p\in\gamma}s_p.$$
Since the foliation $\mathcal S$ is analytic and transverse to
$\mathcal T$, by shrinking $\gamma$ and the leaves $s_p$,
$p\in\gamma$,  we may assume that $S\subset\Sigma$ is a
two-dimensional real-analytic submanifold. Moreover,
$S\subset\CC{n+1}$ is totally real (since $l(p),c(p)$ lie in
complementary complex subspaces of $\CC{n+1}$). Thus there exists
a biholomorphic transformation $(\CC{n+1},0)\mapsto (\CC{n+1},0)$,
transforming $S$ into the totally real plane
\begin{equation}
\label{flatn}\Pi^n=\{\check z=0, \re z_n=0, v=0\},
\end{equation}
$\gamma$ into the line
\begin{equation}
\label{Gamman}\Gamma^n=\{z=0, v=0\}\subset\Pi^n,
\end{equation}
and each $s_p,\,p\in\gamma$, into
\begin{equation}
\label{horizontaln}\{\check z=0,\,\re
z_n=0,\,w=b\},\,p=(0,...,0,b),\,b\in\RR{}.
\end{equation}
It is not difficult to see that this transformation can be chosen
in such a way that it preserves \eqref{initialn}. The transformed
hypersurface $M^*$ contains $\Gamma^n$ and thus satisfies, in
addition to \eqref{initialn}, the condition $\Phi^*_{00}(u)=0$. In
what follows we consider only transformations, preserving
$\Gamma^n$ and the conditions \eqref{initialn}. It is \it
important to note \rm that in the normal form \eqref{normalformn}
the set $S$ is not necessarily flat, as in \eqref{flatn}. However,
the slope field $c(p)$ remains horizontal for $p\in\Gamma^n$, as
in \eqref{horizontaln}, and the Levi kernels $K_p,\,p\in\Gamma^n$,
all look as $\{\check z=0,\,w=0\}$.

\medskip

\noindent \bf Normalization of Segre varieties along a chain. \rm
Arguing identically to the corresponding step in Section 2, we
perform the unique biholomorphic transformation of the form
$$z^*=z,\quad w^*=g(z,w),\quad g(0,w)=0,$$ transforming $(M,0)$
into a real-analytic hyperfurface, satisfying
$\Phi^*_{k'k00}(u)=0$ for any $k',k\geq 0$.
 It is straightforward that $\Gamma^n$ is preserved and the slope field
$c(p)$ remains horizontal (i.e.\  of the form \eqref{horizontaln})
for $p\in\Gamma^n$. The Segre varieties of points
$(\xi,\eta)\in\CC{n}\times\CC{}$ with $(\xi,\eta)\in\Gamma^n$ are
all of the form $\{w=\bar\eta\}$. In addition, we claim that for
the new hypersurface $M^*$ one has
\begin{equation}
\label{nohermitian}\Phi_{z_j\overline{z_n}}(0,0,u)=0,\quad
j=1,\ldots,n, \quad \text{ i.e. } \quad
\Phi_{1001}=\Phi_{0110}=\Phi_{0101}=0.
\end{equation}
This follows from the fact that the complex tangents  $T_p^{\CC{}}
M,\,p\in\Gamma^n$, all have the form $\{w=0\}$ in the new
coordinates, while the Levi kernels $K_p,\,p\in\Gamma^n$, all
remain $\{\check z=0,\,w=0\}$.

\medskip

\noindent \bf Normalization of the Segre map. \rm This step can be
interpreted as a normalization of the Segre map, considered as an
antiholomorphic map
$$p\mapsto \left(w^p;w^p_{z_1},\ldots,w^p_{z_n};w^p_{z_n^2}\right)\bigl|_{z=0}\bigr.,$$
assigning to a point $p\in\CC{n+1}$ the partial 2-jet at $z=0$ of
its Segre variety $Q_p=\{w=w^p(z)\}$.

Similarly to Section 2, we introduce the subspace
$\6D^n\subset\mathcal F^n$, which consists of all convergent power
series of the form
$$\sum\limits_{k'\geq 2}\bigl(\Psi_{k'001}(\check z,\overline{\check
z},u)\bar z_n +\Psi_{01k'0}(\check z,\overline{\check
z},u)z_n\bigr) +\sum\limits_{ k'+k,\, l'+l\geq
2}\Psi_{k'kl'l}(\check z,\overline{\check z},u)z_n^k\bar z_n^l.
$$
We aim to bring a hypersurface $M$, obtained in the previous step,
to such a form that all terms of weight $\geq 7$ in $\Phi$ belong
to the space $\mathcal D^n$, i.e., to bring the defining equation
to the form
\begin{equation}
\label{prenormaln} v=P(z,\bar z) \mod \6D^n.
\end{equation}
The subspace of terms to be removed from $\Phi$ in the current
step is transverse to $\mathcal D^n$, and adding to $\Phi$ an
element of $\mathcal D^n$ does not change the desired form of it,
that is why it is convenient for us to use identities modulo
$\mathcal D^n$.

We begin by making the terms $\Phi_{1010}(\check
z,\overline{\check z},u)$ and $\Phi_{0201}(u)$ for the
hypersurface $M$, obtained in the previous step, independent of
$u$. For that we note that $\Phi_{1010}(\check z,\overline{\check
z},u)$ is an analytic family $H_u(\check z,\overline{\check z})$
of Hermitian forms, and $H_0(\check z,\overline{\check z})=\langle
\check z,\overline{\check z}\rangle$ is nondegenerate. We denote
the  analytic function $\Phi_{0201}(u)$ as $a(u)$, $a(0)=1$. Then
by the implicit function theorem, there exist real-analytic
functions $T(u)\in{\sf GL}(n-1,\C)$ and $c(u)\in\C^*$
 near the origin such that
$$H_u(\check z,\overline{\check
z})=\langle T(u)\check z,\overline{T(u)\check z}\rangle,\quad
a(u)=(c(u))^2\overline{c(u)},\quad T(0)=\mbox{Id},\,c(0)=1.$$ Then
we perform the biholomorphic transformation
\begin{equation}
\label{killun} \check z\mapsto T(w)\check z,\quad z_n\mapsto
c(w)z_n,\quad w\mapsto w.
\end{equation}
Since for the initial hypersurface the defining function
$\Phi(z,\bar z,u)$ contains only terms of the form
$\Phi_{k'kl'l},\,k'+k\geq 1,\,l'+l\geq 1$, the same property holds
for the new defining function $\Phi^*$ of the new hypersurface
$M^*$, and we compute
$$\Phi(z,\bar z,u)=\Phi^*(z^*,\bar z^*,u^*)=\Phi^*(\check zT(u),c(u)z_n,\overline{\check zT(u)},\overline{c(u)z_n},u)\quad
\mod \6D^n,$$ provided $(z,w)\in M$.  It follows from here that
 $$\Phi_{1010}(\check z,\overline{\check z},u)=\Phi^*_{1010}(\check zT(u),\overline{\check zT(u)},u)$$
 and
 $$\Phi_{0201}(u)=\Phi^*_{0201}(u)(c(u))^2\overline{c(u)},$$
 so that
$$\Phi^*_{1010}(\check z^*,\overline{\check z^*},u^*)=\langle\check z^*,\overline{\check z^*}\rangle,\quad
\Phi^*_{0201}(u^*)=1,$$ as required.

We then achieve the condition \eqref{prenormaln} by performing a
transformation
$$z^*=z+f(z,w),\quad w^*=w, \quad f(z,w)=O(|z|^2),$$ where $f$
will be determined below.
We do so in several steps.

First, we eliminate terms $\Phi_{k'k10}$ with $k'+k\geq 2$ by a
transformation
$$
(\check z, z_n, w)\mapsto (\check z+\check f(z,w),z_n, w), \quad
\check f(z,w)=O(|z|^2).$$ For that, we choose a holomorphic
function $B(z,u)=O(|z|^2)\in\C^{n-1}$ satisfying
$$
\sum_{k'+k\ge2}\Phi_{k'k10}(\check z,\1{\check z},u)z_n^k= \langle
B(z,u),\overline{\check z}\rangle.$$ Then, arguing as above and
using identities modulo $\mathcal D^n$, we see that the
transformation
$$(\check z, z_n, w)\mapsto (\check z + B(z,u), z_n, w)$$
transforms $M$ into a hypersurface satisfying
$$\Phi^*_{k'k10}=0, \quad k'+k\ge 2.$$


Second, we use similar arguments to remove the terms
$\Phi_{k'101}$ with $k'\geq 1$ by a transformation
$$(\check z,z^{n}, w)\mapsto \bigl(\check z, z^n+f^n(\check
z,w), w\bigr),\quad f^n=O(|\check z|^2).$$ Thanks to
\eqref{nohermitian}, we expand the defining function of the
hypersurface, obtained in the previous step, as
\begin{gather*}
\Phi(z,\bar z,u)=\langle \check z,\overline{\check
z}\rangle+(z^n+\varphi(\check z,u))^2\overline{(z^n+\varphi(\check
z,u))}+(z^n+\varphi(\check z,u))(\overline{z^n+\varphi(\check
z,u)})^2+\\
+(z^n)^2\overline{z^n}\cdot O(|z|)+z^n(\overline{z^n})^2\cdot
O(|z|)\quad (\mbox{mod}\,\mathcal D^n)
\end{gather*}
for an appropriate analytic function $\varphi(\check
z,u)=O(|\check z|^2)$ (more precisely, we choose
$\varphi:=\frac{1}{2}\sum_{k'\geq 2} \Phi_{k'101}$). Then putting
$f^n:=\varphi(\check z,w)$, we obtain the desired transformation.
%

Finally, we remove the terms $\Phi_{kl01}$, $l\geq 2,\,k+l\geq 3$.
We expand the defining function of the hypersurface, obtained in
the previous step, as
\begin{gather*}
\Phi(z,\bar z,u)=\langle \check z,\overline{\check
z}\rangle+(z_n+z_n\psi(z,u))^2\overline{(z_n+z_n\psi(z,u))}+\\
+(z_n+z_n\psi(z,u))(\overline{z_n+z_n\psi(z,u)})^2 \quad \mod
\6D^n
\end{gather*}
for an appropriate real-analytic function $\psi(z,u)=O(|z|)$,
where one determines $\psi$ by the implicit function theorem from
the equation
$$2\psi+\psi^2=
\sum_{k+l\geq 3,\,l\geq 2} \Phi_{kl01}(\check z, 0,u)z_n^{l-2}.$$
Then we achieve \eqref{prenormaln} by performing the
transformation
$$
(\check z,z_n,w)\mapsto \bigl(\check z,
z_n+z_n\psi(z,w),w\bigr).$$

 The Segre map is now given by
\begin{equation}
\label{Segren}s=(\xi,\eta)\mapsto
\left(\bar\eta;2i\varepsilon^1\overline{\xi_1},...,2i\varepsilon^{n-1}\overline{\xi_{n-1}},2i(\overline{\xi^n})^2+O(|\check\xi|^2);4i\overline{\xi^n}+O(|\xi|^2)\right).
\end{equation}
Clearly, $\Gamma^n$ is preserved and the Levi kernels
$K_p,\,p\in\Gamma^n$, all remain of the form $\{\check
z=0,\,w=0\}$.

\medskip

\noindent \bf  Fixing an orthonormal basis in
 the Levi-nondegenerate direction
along a chain. \rm We next achieve the normalization condition
$$\re\Phi_{1211}=0$$
 by means of a transformation
\begin{equation}
\label{killre} (\check z, z_n, w)\mapsto (C(w)\check z, z_n,w)
\end{equation}
for an appropriate  holomorphic near the origin function $C(w)$
such that $C(w)\in U(r,n-1-r)$ for $w\in\RR{}$ and
$C(0)=\mbox{Id}$. This step can be interpreted as an analytic
choice of an orthonormal basis
$\bigl\{e^1(u_0),...,e^{n-1}(u_0)\bigr\}$ with respect to the form
$\langle\check z,\overline{\check z}\rangle$ in the
Levi-nondegenerate direction $K^T_p$ at every point
$p=(0,u_0)\in\Gamma^n$.

Let us introduce the subspace $\mathcal C^n\subset\mathcal D^n$
(where $\mathcal D^n\subset\mathcal F^n$ is the space  of power
series used in the previous step), which consists of elements
$\Phi\in\mathcal D^n$ satisfying $\Phi_{1211}=0.$
  It is again convenient for us to use
identities modulo $\mathcal C^n$. Let us then fix some analytic
function $C(u)$, valued in $U(r,n-1-r)$. Note that for any fixed
$a,b\in\CC{n-1}$ we have
$$0=\Bigl(\langle C(u)a
,\overline{C(u)b}\rangle\Bigr)'=\langle C'(u)a
,\overline{C(u)b}\rangle+\langle C(u)a
,\overline{C'(u)b}\rangle.$$

Thus, after the transformation \eqref{killre}, we obtain (recall
that $P=P(z,\bar z)$ denotes the leading polynomial of $\Phi$)

\begin{gather*}
\Phi(z,\bar
z,u)=\Phi^*(z^*,\overline{z^*},u^*)=2\re(z_n^2\overline{z_n})+2\re\bigl(\Phi^*_{1211}z_n^2\overline{z_n}\bigr)+\\
+\bigl\langle(C(u)+iC'(u)P)\check z,\overline{(C(u)+iC'(u)P)\check
z}\bigr\rangle\quad\mbox{mod}\,\mathcal
C^n=\\
=P(z,\bar
z)+2\re\bigl(\Phi^*_{1211}z_n^2\overline{z_n}\bigr)+P\bigl\langle
iC'\check z,\overline{C\check z}\bigr\rangle+P\bigl\langle C\check
z,-i\overline{C'\check
z}\bigr\rangle\quad\mbox{mod}\,\mathcal C^n =\\
=P(z,\bar
z)+2\re\bigl(\Phi^*_{1211}z_n^2\overline{z_n}\bigr)+2\bigl(z_n^2\bar
z_n+z_n(\overline{z_n})^2\bigr)\bigl\langle iC'\check
z,\overline{C\check z}\bigr\rangle\quad\mbox{mod}\,\mathcal C^n,
\end{gather*}
\noindent provided $(z,w)\in M$. Thus we have
$$\Phi_{1211}(\check z,\overline{\check z},u)=\Phi^*_{1211}(\check z,\overline{\check z},u)+2\bigl\langle iC'\check
z,\overline{C\check z}\bigr\rangle,$$ and the condition
$\re\Phi^*_{1211}=0$ amounts to representing the Hermitian form
$M(\check z,\overline{\check z}):=\frac{1}{2}\re\Phi_{1211}$ in
the form $\bigl\langle iC'\check z,\overline{C\check
z}\bigr\rangle$ for an appropriate function $C(u)\in U(r,n-1-r)$.
Since the Hermitian form $\langle\cdot,\cdot\rangle$ is
nondegenerate, the condition
$$M(\check z,\overline{\check z})=\bigl\langle iC'\check z,\overline{C\check
z}\bigr\rangle$$ can be read as a first order nonsingular analytic
ODE for $C(u)\in\mbox{Mat}(n-1,\CC{})$ near the point
$u=0,C=\mbox{Id}$, that we solve with the initial condition
$C(0)=\mbox{Id}$. It remains to show that for $C(u)$ chosen in
this manner we indeed have $C(u)\in U(r,n-1-r)$. This can be
argued as follows. As $M(\check z,\overline{\check z})$ is a
Hermitian form, we have
$$
\langle C'a ,\overline{Cb}\rangle+\langle
Ca,\overline{C'b}\rangle=0\quad\forall a,b\in\CC{n-1},$$ thus
$$\Bigl(\langle Ca
,\overline{Cb}\rangle\Bigr)'=\langle C'a
,\overline{Cb}\rangle+\langle Ca,\overline{C'b}\rangle=0$$ for any
fixed $a,b\in\CC{n-1}$, so that
$$\langle C(u)a ,\overline{C(u)b}\rangle=\langle C(0)a
,\overline{C(0)b}\rangle=\langle a ,\overline{b}\rangle,$$ as
required.

\medskip

\noindent \bf  Fixing a parametrization for a chain. \rm For the
hypersurface, obtained in the previous step, we claim that
$$\Phi_{1102}(\check z,\overline{\check z},u)=\tr\Phi_{1111}(\check z,\overline{\check z},u)=0$$
in \eqref{prenormaln}.

Indeed, consider the (formal) transformation $F=(f,g)=(\check
f,f^n,g)$, bringing a hypersurface \eqref{prenormaln} into normal
form, and study the equation \eqref{tangencyn}, applied to it.
Collecting terms with $z_n\overline{z_n}u^l,\,l\geq 0$, we first
obtain $\re f^n_{00}(u)=0$. Collecting all terms of the form
$\check z^\a(\overline{\check z})^\beta
z_n\overline{z_n}u^0,\,|\a|=|\beta|=1$, we obtain precisely the
equation \eqref{findfn00}, where we substitute $u:=0$ and
$\Psi:=2\Phi$. Since $\Gamma^n\subset M$ is a degenerate chain, we
have $(f^n_{00})'(0)=0$, so that $\tr\Phi_{1111}(\check
z,\overline{\check z},0)=0$.

Next, collecting first in \eqref{tangencyn} all terms of the form
$\check z^\a u,\,|\a|=1$, and second all terms of the form $\check
z^\a\overline{z_n}(z_n)^2u^0,\,|\a|=1$, it is straightforward to
check that we obtain precisely the equations \eqref{starting},
where the first equation is differentiated and substituted
$u:=0,\,\Psi_{1000}:=0$, and the second is substituted
$u:=0,\,\Psi_{1102}:=2\Phi_{1102}$. Thanks to the fact that
$\Gamma^n$ is a degenerate chain we have $\check f'_{00}(0)=0$,
and we conclude that $\Phi_{1102}(\check z,\overline{\check
z},0)=0$. Since the prenormal form \eqref{prenormaln} is invariant
under the real shifts $w\mapsto w+u_0,\,u_0\in\RR{}$, and
$\Gamma^n\subset M$ is a degenerate chain, we similarly conclude
that
$$\Phi_{1102}(\check z,\overline{\check z},u_0)=\tr\Phi_{1111}(\check z,\overline{\check z},u_0)=0$$
for any small $u_0\in\RR{}$ in \eqref{prenormaln}, as required.

It remains to achieve the last normalization condition
$$\im\tr\Phi_{1211}=0,$$
using, as before, an appropriate gauge transformation
\begin{equation}
\label{gaugen} \check z\mapsto \check f(w)\check z,\quad
z_n\mapsto f^n(w)z_n,\quad w\mapsto g(w),\quad \check f(0)\neq
0,\,f^n(0)\neq 0,\,g(0)=0,\,g'(0)\neq 0.
\end{equation}
We do so by choosing a gauge transformation
$$\check z^*=(q'(u))^{1/2}\check z,\quad {z_n}^*=(q'(w))^{1/3}z_n,\quad w^*=q(w)$$ for an appropriate
$q(w)$ with $\im q(u)=0,\,q(0)=0,q'(0)>0$. This transformation can
be interpreted as a choice of parametrization for the initial
degenerate chain $\gamma$.

We introduce the subspace $\mathcal A^n\subset\mathcal C^n$,
characterized by the conditions
$$\Phi_{1102}=\Phi_{1111}=0.$$ For the same reason as
in the previous steps, we use identities modulo $\mathcal A^n$.
Note also that $\mathcal A^n$ is contained in the space $\mathcal
N^n$ of series, satisfying the normalization conditions
\eqref{normalformn}.

Applying the above transformation to a hypersurface, satisfying
\eqref{prenormaln} and
$$\Phi_{1102}=\re\Phi_{1211}=\tr\Phi_{1111}=0,$$
it is straightforward to compute (as before, we use expansions of
the form $h(u+iv)=h(u)+ih'(u)v+\dots$)
\begin{gather*}
v^*=q'(u)v\quad\mbox{mod}\,\mathcal
A^n,\\
P(z^*,\overline{z^*})=q'(u)P(z,\bar
z)+\frac{i}{3}q''(u)\langle\check z,\overline{\check
z}\rangle(z_n)^2\overline{z_n}\quad\mbox{mod}\,\mathcal
A^n\\
v^*-P(z^*,\overline{z^*})=q'(u)(v-P(z,\bar
z))+\frac{i}{3}q''(u)\langle\check z,\overline{\check
z}\rangle\left((z_n)^2\overline{z_n}-z_n(\overline{z_n})^2\right)\quad\mbox{mod}\,\mathcal
A^n.
\end{gather*}
Hence
$$\im\tr\Phi^*_{1211}=q'\cdot\im\tr\Phi_{1211}+\frac{2(n-1)}{3}q'',$$
so that the condition $\im\tr\Phi^*_{1211}=0$ becomes a second
order nonsingular ODE for $q(u)$. Solving it with some initial
condition $q'(0)>0$, we finally obtain a hypersurface, satisfying
all the normalization conditions \eqref{normalformn}. It is not
difficult to see that the gauge transformation chosen to achieve
$\im\tr\Phi^*_{1211}=0$ must have the above form and hence is
unique up to the choice of the real parameter $q'(0)$,
corresponding to the action of the subgroup
\begin{equation}
\label{scalingsn} \check z\mapsto \lambda^3\check z,\quad
z_n\mapsto \lambda ^2 z_n,\quad w\mapsto \lambda^6
w.
\end{equation}
in \eqref{dilationsn}. Thus \it

\smallskip

we can canonically, up to the action of the group
\eqref{scalingsn}, choose a parametrization on each degenerate
chain.\rm
\medskip

Theorem 2 is completely proved now. As in the two-dimensional
case, we can see from the proof that in the normal form
\eqref{normalformn} the unique degenerate chain, passing through
the origin, is given by \eqref{Gamman}.







\end{document}